\documentclass[a4paper,11pt,reqno]{amsart}
\usepackage{amsmath,amssymb,amsthm}
\usepackage{mathrsfs}
\usepackage{graphics}
\usepackage{graphicx}
\usepackage{epsfig}
\theoremstyle{plain}
\newtheorem{thm}{Theorem}[section]
\newtheorem{lemma}[thm]{Lemma}
\theoremstyle{definition}
\newtheorem{defn}{Definition}[section]
\theoremstyle{remark}

\numberwithin{equation}{section}
\begin{document}
\title[The fractional S-transform on BMO and Hardy spaces]{The fractional S-transform on BMO and Hardy spaces}
\author[B. Kalita and S. K. Singh]{Baby Kalita$^1$ and  Sunil Kumar Singh$^{2,\ast}$}
\date{}
\begin{abstract} In this paper, the authors studied the fractional S-transform on BMO and Hardy spaces and generalized the results given in \cite{Singh15h}. In introduction section, the definitions of the S-transform, fractional Fourier transform and fractional S-transform are given. In Section 2 continuity and boundedness results for the fractional S-transform in BMO and Hardy spaces are obtained. Furthermore, in Section 3 the fractional S-transform is studied on the  weighted BMO$_{\kappa}$ and weighted Hardy spaces associated with a tempered weight function. 
\end{abstract}
\maketitle
\thispagestyle{empty}
\noindent \textbf{Key words}: S-transform, Fractional S-transform, BMO space, Hardy space \\
\textbf{2010 AMS Subject Classification}: 65R10; 32A37; 30H10\\
\section{Introduction}
The S-transform is a new time-frequency analysis method, which is deduced from short-time Fourier transform and an extension of wavelet transform. It was first introduced by Stockwell et al.\cite{Stockwell96} in 1996 for analyzing geophysics data and since then has been applied in several discipline, such as geophysics, oceanography, atmospheric physics, medicine, hydrogeology. The continuous S-transform of a function $f$ with respect to the window function $\omega$ is defined as \cite{Ventosa08}
\begin{equation}
\label{eq:Int-1}
 (S_\omega f)(\tau,\xi)=\int_{\mathbb{R}} f(t) \, \omega(\tau-t,\xi) \, e^{-i2\pi \xi t} \, dt, 
 ~\text{for}~ \tau,\xi \in \mathbb{R},
\end{equation}
where the window $\omega$ is assumed to satisfy the following:
\begin{equation}
\int_{\mathbb{R}} \omega(t,\xi)dt=1 ~ \text{for all}~ \xi \in \mathbb{R}_0 := \mathbb{R} \setminus \{0\}.
\end{equation}
The most usual window $\omega $ is a Gaussian one
 \begin{equation}
 \omega(t,\xi)= \frac{\left|\xi\right|}{k\sqrt{2\pi}}\,e^{-\xi^2t^2/2k^2}, k>0,
\end{equation}
where $\xi$ is the frequency, $t$ is the time variable, and $k$ is a scaling factor that controls the number of oscillations in the window.\par
Equation(\ref{eq:Int-1}) can be rewritten as a convolution 
\begin{equation}
\label{eq:Int-2}
 (S_\omega f)(\tau,\xi) = \left(f(\cdot)e^{-i2\pi \xi \cdot}\ast\omega(\cdot,\xi)\right)(\tau).
\end{equation} 
Applying the convolution property for the Fourier transform in (\ref{eq:Int-2}), we can obtain a direct relation between S-transform and Fourier transform as follows:
\begin{equation}
\label{eq:Int-3}
 (S_\omega f)(\tau,\xi) = \mathscr{F} ^{-1}\left\{\hat{f}(\cdot +\xi)~\hat{\omega}(\cdot,\xi)\right\}(\tau),
\end{equation}
where $\hat {f}(\eta) = (\mathscr{F}f)(\eta)= \int_{\mathbb{R}} f(t)~ e^{-i2\pi \eta t} dt$ is the  Fourier transform of $f$. Certain examples and basic properties of the S-transform can be found in \cite{Singh12s,Singh12w,Singh12t,Singh13s,Singh13w,Singh13ss,Singh13d,Singh15b}.
\subsection{The fractional Fourier transform}
The fractional Fourier transform was introduced by Almeida \cite{Almeida94} as a generalization of classical Fourier transform. It has several known applications in the signal processing and many other scientific fields. The $a^{th}$ order FRFT (fractional Fourier transform) of a signal $f$ is defined as
\begin{equation}
\label{eq:2.1}
F^a_f(\xi)= \int_{\mathbb{R}} f(t) \, K_a(t,\xi) \, dt,
\end{equation}
where the transform kernel is given by 
\begin{equation}
\label{ffk}
K_a(t,\xi)=
\begin{cases}
A_{\theta}\,e^{i{\pi}({\xi}^2 \cot \theta - 2\xi t \csc \theta +t^2 \cot \theta)},& \text{if $\theta \neq j\pi$}\\
\delta(t-\xi),& \text{if $\theta = 2j\pi$}\\
\delta(t+\xi),& \text{if $\theta+\pi =2j\pi$},
\end{cases}
\end{equation}
where $A_{\theta}= \sqrt{1-i\cot \theta}, \theta= a\frac{\pi}{2}$, $a \in [0,4)$, $i$ is the complex unit, $j $ is an integer and $\xi$ is the fractional Fourier frequency. The inverse FRFT of equation (\ref{eq:2.1}) is 
\begin{equation*}
f(t)= \int_{\mathbb{R}}F^a_f(\xi) \overline{K_a(t,\xi)} d\xi.
\end{equation*}
\subsection{The fractional S-transform}
The fractional S-transform is used by Xu et.al.\cite{Xu12} in 2012 as a generalization of the S-transform. The $a^{th}$ order continuous fractional S-transform (FRST) of a signal $f(t)$ is defined as
\begin{equation}
FRST^a_f(\tau,\xi)= \int_{\mathbb{R}} f(t) g(\tau-t,\xi) K_a(t,\xi) dt,
\end{equation}
where the window function $g$ is given by 
\begin{equation}
\label{eq:3.1}
g(t,\xi)= \frac{\left|\xi \csc \theta\right|^p}{k\sqrt{2\pi}} e^{-{t^2(\xi \csc \theta})^{2p}/2k^2}~~; \,k,p >0,
\end{equation}
which satisfies the condition:
\begin{equation}
\label{eq:3.2}
\int_{\mathbb{R}}g(t,\xi)dt=1 , ~\text{for all}~ \xi \in \mathbb{R}_0.
\end{equation}
Inverse fractional S-transform is defined as 
\begin{equation*}
f(t)=\int_{\mathbb{R}}\left[ \int_{\mathbb{R}}FRST^a_f(\tau,\xi) \, d\tau \right]\overline{K_a(t,\xi)} d\xi.
\end{equation*}
The fractional S-transform depends on a parameter $\theta$ and can be interpreted as a rotation by an angle  $\theta$ in the time-frequency plane. The parameter $p$ and $k$ can be used to adjust the window function space. The fractional S-transform has been studied on distribution spaces by Singh \cite{Singh12t, Singh13s, Singh13w}.
\section{The fractional S-transform  on BMO and Hardy  spaces}
The bounded mean oscillation space $BMO(\mathbb{R})$ also known as John-Nirenberg space was first introduced by F. John and L. Nirenberg in 1961\cite{John}. It is the dual space of the real Hardy space $H^1$ and serves in many ways as a substitute space for $L^\infty$. The $BMO(\mathbb{R})$ spaces play an important role in various areas of mathematics, for example many operators which are bounded on $L^p$, $1 < p < \infty$, but not on $L^\infty$, are bounded when considered as operators on $BMO$.\par
Now, we recall the definitions of the BMO and Hardy spaces. 
\begin{defn}
The bounded mean oscillation space $BMO(\mathbb{R})$ is defined in \cite[p. 140]{Stein93} as the space of all  Lebesgue integrable (locally) functions defined on  $\mathbb{R}$ such that 
\begin{equation}
\parallel f\parallel _{BMO}= \sup_{I\subset{\mathbb{R}}} {\frac{1}{|I|}}\int _I |f(x)-f_{I}|\,dx < \infty,
\end{equation}
here the supremum is taken over all intervals $I$ in $\mathbb{R}$ of measure $|I|$ and $f_I$ stands for the mean of $f$ on $I$, namely 
\begin{equation}
\label{eq:Int-4.2}
f_{I}:=\frac{1}{|I|}\int_I f(x)\,dx \leq \frac{1}{|I|}\int_I |f(x)|\,dx \leq \,m <\, \infty.
\end{equation}
\end{defn}
\begin{defn}
\label{def:2.2}
The Hardy space is defined in \cite[pp. 90-91]{Stein93} as the space of all functions $ f \in L^1( \mathbb{R})$ such that
\begin{equation}
\label{eq:5.1}
\parallel f\parallel _{H^1}=\int_{\mathbb{R}} \sup_{t>0} |\left(f\ast\phi_t\right)(x)|\,dx < \infty,
\end{equation}
where $\phi$ is any test function with $\int _{\mathbb{R}} \phi(x) dx \neq 0$ and 
$\phi_t(x)=t^{-1}\phi(x/t)$;  $t>0$, $x\in {\mathbb{R}}$.
\end{defn}
Our main results in this section are as follows.
\begin{lemma}
If $K_a(x,\xi)$ is the transform kernel defined in (\ref{ffk}) and $\theta \neq j\pi$, then for any Lebesgue integrable (locally) function defined on $\mathbb{R}$ we have
\begin{equation*}
\left\|f(\cdot) K_a(\cdot, \xi)\right\|_{BMO}\leq |A_\theta| \left(\left\|f\right\|_{BMO}+2m\right)
\end{equation*}
where $m$ is a constant given in equation (\ref{eq:Int-4.2}).
\end{lemma}
\begin{proof}
By using the inequality $|K_a(x,\xi)|\leq |A_\theta|$, we have
\begin{equation*}
\begin{split}
&\left\|f(\cdot)K_a(\cdot, \xi)\right\|_{BMO}\\
&=\sup_{I\subset{\mathbb{R}}}\frac{1}{|I|}\int_{I} \bigg| f(x) K_a(x, \xi)-\frac{1}{|I|}\int_{I} f(t) K_a(t, \xi) dt\bigg|dx\\
&=\sup_{I\subset{\mathbb{R}}}\frac{1}{|I|}\int_{I} \bigg| f(x) K_a(x, \xi)-\frac{K_a(x, \xi)}{|I|}\int_I f(t) dt 
+ \frac{K_a(x, \xi)}{|I|}\int_I f(t) dt\\
& ~ -\frac{1}{|I|}\int_{I} f(t) K_a(t, \xi)  dt\bigg|dx\\
&\leq  \sup_{I\subset{\mathbb{R}}}\frac{1}{|I|}\int_{I}\left| K_a(x, \xi)\left(f(x)-\frac{1}{|I|}\int_I f(t) dt\right)\right|dx\\
& ~+\sup_{I\subset{\mathbb{R}}}\frac{1}{|I|}\int_{I}\left|\frac{K_a(x, \xi)}{|I|}\int_I f(t) dt\right|dx
 +\sup_{I\subset{\mathbb{R}}}\frac{1}{|I|}\int_{I}\left|\frac{1}{|I|}\int_{I} f(t) K_a(t, \xi) dt\right|dx\\
& \leq  |A_\theta|\left\|f\right\|_{BMO} + |A_\theta| \sup_{I\subset{\mathbb{R}}}\frac{1}{|I|}\int_{I}\left|f\right|_I dx
 + |A_\theta| \sup_{I\subset{\mathbb{R}}}\frac{1}{|I|}\int_{I} \frac{1}{|I|}\int_{I} |f(t)| dt dx\\
&\leq  |A_\theta|\left(\left\|f\right\|_{BMO} + \frac{1}{|I|}m|I| + \sup_{I\subset{\mathbb{R}}}\frac{1}{|I|}\int_{I} m \, dx \right)\\
&\leq |A_\theta|\left( \left\|f\right\|_{BMO}+\frac{1}{|I|}m|I|+\frac{1}{|I|}m|I|\right)\\
&= |A_\theta|\left(\left\|f\right\|_{BMO} + 2m\right).
\end{split}
\end{equation*}
\end{proof}

\begin{thm}
\label{thm-4.1}
 For any fixed  $\xi \in \mathbb{R}_0$ and $\theta \neq j\pi$, the operator $FRST^a_f: BMO(\mathbb{R})\to BMO(\mathbb{R})$ is continuous and 
\begin{equation*}
\parallel (FRST^a_f)(\cdot,\xi)\parallel _{BMO} ~\leq ~ \left|A_{\theta}\right|\left(\left\|f\right\|_{BMO}+2m\right).
\end{equation*}
\end{thm}
\begin{proof} Since
\begin{equation*}
\begin{split}
&\left|(FRST^a_f)(\tau,\xi)-(FRST^a_f)_{I}(\tau,\xi)\right|\\
&=\bigg|\int_{\mathbb{R}} f(\tau-x) g(x,\xi) K_a(\tau-x,\xi) dx\\
&~ - \frac{1}{|I|}\int_{I}\int_{\mathbb{R}} f(\alpha-x) g(x,\xi) K_a(\alpha-x,\xi) dx \,d\alpha\bigg| \\
&=\bigg|\int_{\mathbb{R}} f(\tau-x) g(x,\xi) K_a(\tau-x,\xi) dx\\
&~-\int_{\mathbb{R}} g(x,\xi)\left(\frac{1}{|I|}\int_{I} f(\alpha-x) K_a(\alpha-x,\xi)  d\alpha\right) dx \bigg| \\
&=\left|\int_{\mathbb{R}} g(x,\xi)\bigg(f(\tau-x)  K_a(\tau-x,\xi)-\frac{1}{|I|}\int_{I} f(\alpha-x) K_a(\alpha-x,\xi)  d\alpha\bigg)dx\right|\\
& \leq \int_{\mathbb{R}} |g(x,\xi)| \left|f(\tau-x)  K_a(\tau-x,\xi)-\frac{1}{|I|}\int_{I} f(\alpha-x) K_a(\alpha-x,\xi)  d\alpha\right|dx.\\
\end{split}
\end{equation*}
Therefore,
\begin{equation*}
\begin{split}
&\parallel (FRST^a_f)(\cdot,\xi)\parallel_{BMO}\\
&=\sup_{I\subset{\mathbb{R}}}\frac{1}{|I|}\int_{I}\left|(FRST^a_f)(\tau,\xi)-(FRST^a_f)_{I}(\tau,\xi)\right|d\tau \\
&\leq \sup_{I\subset{\mathbb{R}}}\frac{1}{|I|} \int_{I} \int_{\mathbb{R}}|g(x,\xi)| \bigg|f(\tau-x)  K_a(\tau-x,\xi)\\
&~-\frac{1}{|I|}\int_{I} f(\alpha-x) K_a(\alpha-x,\xi)  d\alpha\bigg|dx \, d\tau\\
&= \int_{\mathbb{R}}|g(x,\xi)|  \bigg( \sup_{I\subset{\mathbb{R}}}\frac{1}{|I|} \int_{I} \bigg|f(\tau-x)  K_a(\tau-x,\xi)\\
&~-\frac{1}{|I|}\int_{I} f(\alpha-x) K_a(\alpha-x,\xi)  d\alpha\bigg| d\tau\bigg) dx\\
&= \int_{\mathbb{R}}|g(x,\xi)|  \left( \sup_{J \subset{\mathbb{R}}}\frac{1}{|J|} \int_{J} \left|f(t)  K_a(t,\xi)-\frac{1}{|J|}\int_{J} f(y) K_a(y,\xi)  dy\right| dt \right) dx\\
& \hspace{1cm} (\text{here}\, J=I-x\, \text{for} ~ x \in \mathbb{R})\\
&= \left\|f(\cdot) K_a(\cdot, \xi)\right\|_{BMO} \int_{\mathbb{R}}|g(x,\xi)| dx\\
&= \left\|f(\cdot) K_a(\cdot, \xi)\right\|_{BMO}. \\
\end{split}
\end{equation*}
Now by using above lemma we get
\begin{eqnarray*}
\parallel (FRST^a_f)(\cdot,\xi)\parallel_{BMO}
&\leq &\left|A_{\theta}\right|\left(\left\|f\right\|_{BMO}+2m\right).
\end{eqnarray*}
\end{proof}
\begin{thm}
\label{thm 5.1}
Let $f \in L^1({\mathbb{R}})$ such that
\begin{equation}
\label{eq:2.4}
\sup_{t>0}\left|\int_{\mathbb{R}} f(x-y) \phi_t(y)dy\right|= \sup_{t>0}\int_{\mathbb{R}}\left|f(x-y) \phi_t(y)\right|dy < \infty.
\end{equation}
Then for any fixed $\xi\in \mathbb{R}_0 $, the operator $FRST^a_f: H^1({\mathbb{R}}) \to H^1({\mathbb{R}})$ is continuous and
\begin{equation*}
\parallel (FRST)(\cdot,\xi)\parallel _{H^1} \leq \left|A_{\theta}\right| \parallel f\parallel_{H^1}.
\end{equation*}
\end{thm}
\begin{proof}
We have
\begin{equation*}
\begin{split}
&\parallel FRST^a_f(\cdot,\xi)\parallel _{H^1}\\
&=\int_{\mathbb{R}} \sup_{t>0} \left|\left(FRST^a_f(\cdot,\xi)*\phi_t(\cdot) \right)(\tau)\right|d\tau\\
&=\int_{\mathbb{R}} \sup_{t>0} \left|\left(\left(\int_{\mathbb{R}} f(\cdot-x) g(x,\xi) K_a(\cdot-x,\xi) dx \right)*\phi_t(\cdot) \right)(\tau)\right|d\tau\\
&=\int_{\mathbb{R}} \sup_{t>0}\left|\int_{\mathbb{R}}g(x,\xi)\left(\int_{\mathbb{R}}f(\tau-x-y)  K_a(\tau-x-y,\xi)\phi_t(y)dy\right)dx\right|d\tau\\
&\leq  \int_{\mathbb{R}}\left|g(x,\xi)\right|\left(\int_{\mathbb{R}}\sup_{t>0}\int_{\mathbb{R}}\left|f(\tau-x-y) K_a(\tau-x-y,\xi) \phi_t(y)\right|dy d\tau\right)dx\\
&=  \int_{\mathbb{R}}\left|g(x,\xi)\right|\left(\int_{\mathbb{R}}\sup_{t>0}\int_{\mathbb{R}}\left|f(z) K_a(z,\xi) \phi_t(\tau-x-z)\right|dz d\tau\right)dx\\
&\leq  \left|A_{\theta}\right|\int_{\mathbb{R}}\left|g(x,\xi)\right|\left(\int_{\mathbb{R}}\sup_{t>0}\int_{\mathbb{R}}\left|f(z) \phi_t(\tau-x-z)\right|dz d\tau\right)dx\\
&=\left|A_{\theta}\right|\int_{\mathbb{R}}\left|g(x,\xi)\right|\left(\int_{\mathbb{R}}\sup_{t>0}\left|\left(f*\phi_t\right)(\tau-x)\right|d\tau\right)dx\\
& \hspace{1 cm}(\text{by using equation} ~(\ref{eq:2.4})~~ \text{and then definition of convolution})\\
&=\left|A_{\theta}\right|\left\|f\right\|_{H^1}\int_{\mathbb{R}}\left|g(x,\xi)\right|dx\\
&=\left|A_{\theta}\right|\left\|f\right\|_{H^1}.
\end{split}
\end{equation*}
\end{proof}
\section{The fractional S-transform on weighted function spaces}
Let us recall the definitions of relevant weighted function spaces.
\begin{defn}
A positive function $\kappa$ defined on ${\mathbb{R}}$ is called a tempered weight function if there exist positive constants $C$ and $N $ such that 
\begin{equation}
\kappa(\xi+\eta)\leq (1+C\vert \xi \vert)^N \kappa (\eta) ~~  \text{for all } ~ \xi,\eta\in\mathbb{R},
\end{equation}
and the set of all such functions $\kappa$ is denoted by $\mathscr{K}$. Certain examples and basic properties of weight function $\kappa$ can be found in \cite{Hormander83}.
\end{defn}
\begin{defn}
 For $1 \leq p < \infty$, the weighted Lebesgue space $L^{p}_{\kappa}(\mathbb{R})$ is defined as the space of all measurable functions $f$ on $\mathbb{R} $ such that 
\begin{eqnarray*}
\parallel f \parallel_{L_{\kappa}^p}=\left(\int_{\mathbb{R}} \vert f(x) \vert ^{p} \,\kappa(x) dx \right)^{1/p} < \infty .
\end{eqnarray*}
\end{defn}
\begin{defn}
The weighted bounded mean oscillation space BMO$_{\kappa} (\mathbb{R})$ is defined as the space of all weighted 
Lebesgue integrable (locally) functions defined on  $\mathbb{R}$ such that 
\begin{equation*}
\parallel f\parallel _{{BMO}_\kappa}= \sup_{I\subset{\mathbb{R}}} {\frac{1}{|I|_\kappa}}\int _I{|f(x)-f_{I}|\,\kappa(x)dx} < \infty,
\end{equation*}
where the supremum is taken over all intervals $I$ in $\mathbb{R}$ and \ $|I|_\kappa=\int_{I}\kappa(x)dx $.
\end{defn}
\begin{defn}
The weighted Hardy space is defined as the space of all functions $ f \in L_\kappa ^1 (\mathbb{R})$ such that
\begin{equation*}
\parallel f\parallel _{H_\kappa^1}=\int_{\mathbb{R}} \sup_{t>0} \left|\left(f \ast\phi_t\right)(x)\right|\kappa(x)dx < \infty.
\end{equation*}
where $\phi$ is any test function with $\int _{\mathbb{R}} \phi(x) dx \neq 0$ and $\phi_t(x)=t^{-1}\phi(x/t)$;  $t>0$, $x\in {\mathbb{R}}$.
\end{defn}
Our main results on weighted function spaces  are as follows.
\begin{lemma}
\label{lemma 3.1}
If $g(x,\xi)$ is defined by (\ref{eq:3.1}) and $\theta \neq j\pi$, then for any fixed $\xi \in \mathbb{R}_0$ we have the following estimate 
\begin{equation*}
\int_{\mathbb{R}} |g(x,\xi)| (1+C|x|)^N dx \leq A_{\xi,N} ,
\end{equation*}
where $N$ is positive constant, and $A_{\xi,N}$ is a constant depend on $\xi$ and $N$.
\end{lemma}
\begin{proof} Using the well-known inequality,
$|x+y|^n \leq 2^{n-1}(|x|^n + |y|^n), ~ \forall \, x, y \in \mathbb{R}, ~ \forall \, n \in \mathbb{N}$, we have
\begin{eqnarray*}
&&\int_{\mathbb{R}} |g(x,\xi)| (1+C|x|)^N dx\\
&& \leq   \int_{\mathbb{R}} |g(x,\xi)| \, 2^{[N]}\left(1+|Cx|^{[N]+1}\right) dx\\
&& =   2^{[N]} \left( 1+  C^{[N]+1}\int_{\mathbb{R}} |g(x,\xi)| |x|^{[N]+1}dx\right)\\
&& =  2^{[N]} \left( 1+ 2 \, C^{[N]+1}\int_0 ^\infty x^{[N]+1} g(x,\xi) dx\right)\\
&& =  2^{[N]} \left( 1+ 2 \, C^{[N]+1} \frac{2^{([N]+ 1) /2} k^{[N]+1}}{2\sqrt{\pi} \, |\xi \csc \theta|^{p([N]+1)}} \Gamma ([N]/2+1)\right)\\
&& \leq  A_{\xi,N}. 
\end{eqnarray*}
\end{proof}
\begin{lemma}
\label{lemma 3.2}
If $K_a(x,\xi)$ is the transform kernel defined in (\ref{ffk}) and $\theta \neq j\pi$, then for any locally Lebesgue integrable function defined on $\mathbb{R}$ we have
\begin{equation*}
\left\|f(\cdot) K_a(\cdot, \xi)\right\|_{BMO_{\kappa}}\leq |A_\theta| \left(\left\|f\right\|_{BMO_{\kappa}}+2m\right)
\end{equation*}
where $m$ is a constant given in equation (\ref{eq:Int-4.2}).
\end{lemma}
\begin{proof}
We have
\begin{equation*}
\begin{split}
&\left\|f(\cdot)K_a(\cdot, \xi)\right\|_{BMO_{\kappa}}\\
&=\sup_{I\subset{\mathbb{R}}}\frac{1}{|I|_{\kappa}}\int_{I} \bigg| f(x) K_a(x, \xi)-\frac{1}{|I|}\int_{I} f(t) K_a(t, \xi) dt\bigg|{\kappa}(x)dx\\
&=\sup_{I\subset{\mathbb{R}}}\frac{1}{|I|_{\kappa}}\int_{I} \bigg| f(x) K_a(x, \xi)-\frac{K_a(x, \xi)}{|I|}\int_I f(t) dt
 +\frac{K_a(x, \xi)}{|I|}\int_I f(t) dt \\ 
& ~~ -\frac{1}{|I|}\int_{I} f(t) K_a(t, \xi)  dt\bigg| \kappa(x)dx\\
&=\sup_{I\subset{\mathbb{R}}}\frac{1}{|I|_{\kappa}}\int_{I} \bigg| K_a(x, \xi)\left(f(x)-\frac{1}{|I|}\int_I f(t) dt\right)
+\frac{K_a(x, \xi)}{|I|}\int_I f(t) dt\\
& ~~ -\frac{1}{|I|}\int_{I} f(t) K_a(t, \xi) dt\bigg| \kappa(x) dx\\
&\leq \sup_{I\subset{\mathbb{R}}}\frac{1}{|I|_{\kappa}}\int_{I}\left| K_a(x, \xi)\left(f(x)-\frac{1}{|I|}\int_I f(t) dt\right)\right|\kappa(x)dx\\
&~+\sup_{I\subset{\mathbb{R}}}\frac{1}{|I|_{\kappa}}\int_{I}\left|\frac{K_a(x, \xi)}{|I|}\int_I f(t) dt\right| \kappa(x)dx\\
&~+\sup_{I\subset{\mathbb{R}}}\frac{1}{|I|_{\kappa}}\int_{I}\left|\frac{1}{|I|}\int_{I} f(t) K_a(t, \xi) dt\right|\kappa(x) \, dx\\
&\leq  |A_\theta|\left\|f\right\|_{BMO_{\kappa}} + |A_\theta| \sup_{I\subset{\mathbb{R}}}\frac{1}{|I|_{\kappa}}\int_{I}\left|f\right|_I{\kappa}(x) dx\\
&~ + |A_\theta| \sup_{I\subset{\mathbb{R}}}\frac{1}{|I|_{\kappa}}\int_{I} \left(\frac{1}{|I|}\int_{I} |f(t)| dt\right) \kappa(x)dx\\
&\leq  |A_\theta|\left(\left\|f\right\|_{BMO_{\kappa}} + \frac{1}{|I|_{\kappa}}m|I|_{\kappa} + \sup_{I\subset{\mathbb{R}}}\frac{1}{|I|_{\kappa}}\int_{I} m \,\kappa(x) \, dx \right)\\
&\leq |A_\theta|\left( \left\|f\right\|_{BMO_{\kappa}}+\frac{1}{|I|_{\kappa}}m|I|_{\kappa}+\frac{1}{|I|_{\kappa}}m|I|_{\kappa}\right)\\
&=  |A_\theta|\left(\left\|f\right\|_{BMO_{\kappa}} + 2m\right).
\end{split}
\end{equation*}
\end{proof}
\begin{thm}
\label{th-6.1}
For any fixed $\xi\in \mathbb{R}_0 $,
 the operator $FRST^a_f: BMO_\kappa(\mathbb{R}) \to BMO_\kappa(\mathbb{R})$ is continuous. Furthermore, we have 
\begin{equation*}
\parallel (FRST^a_f)(\cdot,\xi)\parallel _{{BMO}_\kappa} \leq A_{\xi, N} \left|A_{\theta}\right| \left(\left\|f\right\|_{BMO_\kappa}+2m\right)
\end{equation*}
where m is a constant given in equation (\ref{eq:Int-4.2}).
\end{thm}
\begin{proof}
\begin{equation*}
\begin{split}
&\parallel (FRST^a_f)(\cdot,\xi)\parallel_{{BMO}_\kappa}\\
&=\sup_{I\subset{\mathbb{R}}}\frac{1}{|I|_\kappa}\int_{I}\left|(FRST^a_f)(\tau,\xi)-(FRST^a_f)_{I}(\tau,\xi)\right| \kappa(\tau)d\tau \\
&=\sup_{I\subset{\mathbb{R}}}\frac{1}{|I|_\kappa}\int_{I}\bigg|\int_{\mathbb{R}} g(x,\xi)\bigg(f(\tau-x)  K_a(\tau-x,\xi)\\
&~-\frac{1}{|I|}\int_{I} f(\alpha-x) K_a(\alpha-x,\xi)  d\alpha\bigg)dx\bigg|\kappa(\tau)d\tau\\
&\leq  \int_{\mathbb{R}} |g(x,\xi)| \bigg(\sup_{I\subset{\mathbb{R}}}\frac{1}{|I|_\kappa}\int_{I} \bigg|f(\tau-x) K_a(\tau-x,\xi)\\
&~-\frac{1}{|I|_\kappa}\int_{I} f(\alpha-x) K_a(\alpha-x,\xi)  d\alpha\bigg|k(\tau)d\tau\bigg)dx\\
&= \int_{\mathbb{R}} |g(x,\xi)| \bigg(\sup_{J\subset{\mathbb{R}}}\frac{1}{|J|_\kappa}\int_{J} \bigg|f(y) K_a(y,\xi)
-\frac{1}{|J|_\kappa}\int_{J} f(t)  K_a(t,\xi) dt\bigg|\kappa(x+y)dy\bigg)dx\\
& \hspace{1cm} (\text{here}\, J=I-x\, \text{for}\, x\in \mathbb{R})\\
&\leq \int_{\mathbb{R}} |g(x,\xi)| \bigg(\sup_{J\subset{\mathbb{R}}}\frac{1}{|J|_\kappa} \int_{J} \bigg|f(y) K_a(y,\xi)\\
&~ -\frac{1}{|J|_\kappa}\int_{J} f(t)  K_a(t,\xi) dt\bigg|(1+C|x|)^N \kappa(y)dy\bigg)dx\\
&=  \left\|f(\cdot) K_a(\cdot,\xi)\right\|_{BMO_\kappa} \int_{\mathbb{R}} |g(x,\xi)| (1+C|x|)^N dx.
\end{split}
\end{equation*}
So by using lemma \ref{lemma 3.1}  and \ref{lemma 3.2}, we get
\begin{eqnarray*}
\parallel (FRST^a_f)(\cdot,\xi)\parallel_{BMO_\kappa} \leq A_{\xi, N} \left|A_{\theta}\right| \left(\left\|f\right\|_{BMO_\kappa}+2m\right).
\end{eqnarray*}
\end{proof}

\begin{thm}
\label{thm:6.1}
Let $f \in L^1({\mathbb{R}})$  and satisfies the condition (\ref{eq:2.4}). Then, for any fixed $\xi\in \mathbb{R}_0 $, the operator $FRST^a_f : H_\kappa^1({\mathbb{R}}) \to H_\kappa^1({\mathbb{R}})$ is continuous. Furthermore, we have
\begin{equation*}
\parallel FRST^a_f(\cdot,\xi)\parallel _{H_\kappa ^1} \leq A_{\xi,N} \left|A_{\theta}\right| \left\|f\right\|_{H^1_{\kappa}}.
\end{equation*}
\end{thm}
\begin{proof}
We have
\begin{equation*}
\begin{split}
&\parallel FRST^a_f(\cdot,\xi)\parallel _{H_\kappa ^1}\\
&=\int_{\mathbb{R}} \sup_{t>0} \left|\left(FRST^a_f(\cdot,\xi)\ast \phi_t(\cdot) \right)(\tau)\right|\kappa(\tau)d\tau\\
&=\int_{\mathbb{R}} \sup_{t>0}\left|\int_{\mathbb{R}}g(x,\xi)\left(\int_{\mathbb{R}}f(\tau-x-y)  K_a(\tau-x-y,\xi)\phi_t(y)dy\right)dx\right|\kappa(\tau)d\tau\\
&\leq  \int_{\mathbb{R}}\left|g(x,\xi)\right|\left(\int_{\mathbb{R}}\sup_{t>0}\int_{\mathbb{R}}\left|f(\tau-x-y) K_a(\tau-x-y,\xi) \phi_t(y)\right|\kappa(\tau)dy\, d\tau\right)dx\\
&=  \int_{\mathbb{R}}\left|g(x,\xi)\right|\left(\int_{\mathbb{R}}\sup_{t>0}\int_{\mathbb{R}}\left|f(z-y) K_a(z-y,\xi) \phi_t(y)\right|\kappa(z+x)dz\, dy\right)dx\\
&\leq  \left|A_{\theta}\right|\int_{\mathbb{R}}\left|g(x,\xi)\right| \left(\int_{\mathbb{R}}\sup_{t>0}\int_{\mathbb{R}}\left|f(z-y) \phi_t(y)\right|(1+C|x|)^N \kappa(z)dz \right)dx\\
&=\left|A_{\theta}\right|\int_{\mathbb{R}}\left|g(x,\xi)\right|(1+C|x|)^N \left(\int_{\mathbb{R}}\sup_{t>0}\left|\left(f \ast \phi_t\right)(z)\right|\kappa(z)dz\right)dx\\
& \hspace{1 cm}(\text{by using equation} ~(\ref{eq:2.4})~~ \text{and then definition of convolution})\\
&= \left|A_{\theta}\right|\left\|f\right\|_{H^1_{\kappa}}\int_{\mathbb{R}}\left|g(x,\xi)\right|(1+C|x|)^N dx.
\end{split}
\end{equation*}
Now by using lemma \ref{lemma 3.1} we get
\begin{eqnarray*}
\parallel FRST^a_f(\cdot,\xi)\parallel _{H_\kappa ^1} \leq A_{\xi,N} \left|A_{\theta}\right| \left\|f\right\|_{H^1_{\kappa}}.
\end{eqnarray*}
\end{proof}
 
\bigskip
\address{Department of Mathematics,
Rajiv Gandhi University\\
Doimukh-791112, Arunachal Pradesh, India\\
$^1$ Email: kalitababy478@gmail.com\\
$^2$ Email: sks$\_$math@yahoo.com\\
$^\ast$ Corresponding author\\}
\end{document}